\documentclass{article}
\usepackage{amsmath, amsfonts, amsthm}
\DeclareMathOperator{\cl}{cl}
\DeclareMathOperator{\rank}{rank}
\newtheorem{defin}{Definition}
\newtheorem{proposition}{Proposition}
\newtheorem{theo}{Theorem}
\newtheorem{cor}{Corollary}
\begin{document}
\title{On simple matroids with a unique minimal tropical basis}
\author{Winfried Hochst\"attler\\FernUniversit\"at in Hagen}
\date{}
\maketitle
\begin{abstract}
  In this note we characterize tropical bases as sets of circuits that by
  orthogonality determine the set of cocircuits of a simple matroid.
  Furthermore, we show that any circuit, which itself is closed, must be
  contained in any tropical basis. This yields a characterization of
  simple matroids which have a unique minimal tropical basis, giving a
  solution for a problem posted in~\cite{TF24}.
\end{abstract}
\section{Introduction}
In tropical geometry a set of defining tropical linear forms for a
tropical linear space is called a tropical basis. In our setting these
linear forms correspond to the circuits of a simple matroid. Since we will
give a purely combinatorial definition of tropical bases of simple matroids
we will not further bother with the connections to tropical geometry
and refer the interested reader to ~\cite{Nak23,YY07,TF24}.

In a recent post in the matroid union blog~\cite{TF24}, Tara Fife asked
for a characterization of those simple matroids which have a unique minimal
tropical basis. In this note we will show that any circuit which is
closed must be contained in any minimal tropical basis. Furthermore,
if a tropical basis contains a circuit which is not closed, then it is
not minimal. Moreover, we characterize tropical bases as those sets
of circuits which, purely by orthogonality, define the set of
cocircuits of the simple matroid. This, together with the above, yields a
characterization of simple matroids with a unique minimal tropical basis.

Additionally, we prove a strict generalization of the result from~\cite{Nak23} that binary simple matroids
have unique minimal tropical bases.

We assume familiarity with matroid theory, the standard reference is ~\cite{oxley}.

\section{Tropical bases}
Let $M$ be a simple matroid $M$ and let $\cl$ denote its closure
operator.  A subset $\mathcal{C}' \subseteq \mathcal{C}$ its set of
circuits is a {\em tropical basis of $M$}, if for every set $X \ne
\cl(X)$, there exists $C \in \mathcal{C}'$ sucht that $|C \setminus
X|=1$.  A tropical basis $\mathcal{C}'$ of $M$ is {\em minimal} if for all
$C \in \mathcal{C}'$ the family $\mathcal{C}' \setminus C$ is no
longer a tropical basis.

Let $\mathcal{D}$ denote a family of sets on a ground set $E$ and $A
\subseteq E$. Then $A$ is {\em orthogonal to $\mathcal{D}$} denoted by
$A \bot \mathcal{D}$ if and only if
\[\forall D \in \mathcal{D}: |A \cap D| \ne 1.\]

\begin{proposition}\label{prop:1}
  Let $\mathcal{C}' \subseteq \mathcal{C}$. Then $\mathcal{C}'$ is a
  tropical basis of $M$ if and only if
 \[\bar X \bot \,\mathcal{C}' \Longleftrightarrow \cl(X)=X,\]
where $\bar X= E \setminus X$ denotes the complement of $X$. 
\end{proposition}
\begin{proof}
  Let $\mathcal{C}'$ be a tropical basis. Assume $X$ is closed. Then even for all $C \in \mathcal{C}$
  \[|C \setminus X| = |C \cap \bar X| \ne 1.\] Thus $\bar X \bot
  \mathcal{C}'$. Now assume $X$ is not closed. Since
  $\mathcal{C}'$ is a tropical basis there exists some $C \in
  \mathcal{C}'$ such that $|C \setminus X| = |C \cap \bar X|=1$.
  Hence, $\bar X$ is not orthogonal to $\mathcal{C}'$.

  Now assume $\mathcal{C}'$ is not a tropical basis. Then there exists
  a set $X$ which is not closed but for all $C \in \mathcal{C}'$ we have
\[|C \setminus X|=| C \cap \bar X | \ne 1.\] Hence, $X$ is not closed but $\bar X \bot \mathcal{C}$.

Note that $\cl(X)=X$ implies that $X$ is orthogonal to any set of circuits.
\end{proof}

\begin{proposition}\label{prop:contained}
  Let $C \in \mathcal{C}$ such that $\cl(C)=C$, then $C$ is contained in every tropical basis $\mathcal{C}'$.
\end{proposition}
\begin{proof}
  Let $e \in C$ and $X = C \setminus \{e\}$. The $X \ne \cl(X) \ni e$.
  Hence there exists $C' \in \mathcal{C}'$ and $f \in C'$ such that
  $C' \setminus X=\{f\}$. Now $f \in \cl(X) \subseteq  \cl(C)$ and $\cl(C)=C$ implies $f=e$ and $C'=C.$
\end{proof}
\section{A sufficient condition for the uniqueness of a minimal
  tropical basis}
\begin{defin}
  Let $M$ be a simple matroid on a finite set $E$ and $D \subseteq E$. Then
  $D$ is a {\em double circuit} if for all $e \in D: \rank(D)=|D|-2=
  \rank(D \setminus \{e\})$.
\end{defin}

\begin{proposition}
  If $D$ is a double circuit in $M$, then it has a unique partition $D_1
  \cup \ldots \cup D_k$ such that the set of circuits contained in $D$
  is exactly \[\{D \setminus D_ i \mid 1\le i \le k\}.\]
\end{proposition}
\begin{proof}
  The dual of $D$ is a line. Then the complement of each of its points
  is a cocircuit in the dual Hence we find the $D_i$'s as parallel
  classes of points ot these points.
\end{proof}
We call $k$ the {degree} of the double circuit.
\begin{theo}\label{theo:2}
  Let $M$ be a simple matroid such that for every $C \in \mathcal{C}$ either
  $\cl(C)=C$ or there exist $f \in \cl(C)$ such that $C \cup \{f\}$ is
  a double circuit of degree $3$. Then 
  \[\{C \in \mathcal{C} \mid \cl(C)=C\}\] is the unique minimal tropical
  basis of $M$.
\end{theo}
\begin{proof}
  By Proposition \ref{prop:contained} $\{C \in \mathcal{C} \mid \cl(C)=C\}$
  is contained in any tropical basis. Thus, it suffices to show that it
  is a tropical basis itself. Let $X \subseteq E$ such that $X
  \ne \cl(X)$. Hence there exists a circuit $C \in \mathcal{C}$ such
  that $C \setminus X=\{d\}$. Choose such a circuit with as few elements
  as possible. Assume $\cl(C)\ne C$. Then there exists $f \in \cl(C)
  \setminus C$ such that $C \cup \{f\}=D_1\cup D_2 \cup\{f\}$ is a
  double circuit of degree $3$.  Since $M$ is simple we have
  $\min\{|D_1|,|D_2|\} \ge 2$. We may assume $d \in D_2$. Then, $D_1\cup\{f\}$ is a circuit contradicting the
  choice of $C$.
\end{proof}

\begin{cor}[\cite{Nak23} Theorem 4]
  Binary simple matroids have a unique minimal tropical basis.
\end{cor}
\begin{proof}
  Binary matroids have only double circuits of degree $2$ or $3$. If
  $f \in \cl(C)$ for a circuit $C$, then $C \cup \{f\}$ is a double
  circuit of degree $3$.
\end{proof}

\begin{cor}
  Let $N$ arise from $U_{2,n}$ by replacing each element by a pair of parallel elements and let $M$ denote the dual of $N$. Then $M$ has a unique tropical bases.
\end{cor}

\begin{proof}
  The circuits of $M$ are exactly the complements of the points of $N$
  and hence consist af all elements except one pair of elements
  parallel in $N$. Clearly, they are all closed.
\end{proof}

Since $M$ from the last theorem is non-binary for $n \ge 4$
Theorem~\ref{theo:2} is a proper generalization of \cite{Nak23}
Theorem 4. We are not aware, though, of a non-binary matroid which
satisfies the assumptions of Theorem~\ref{theo:2} and has a circuit
which is not closed.
\section{ A necessary condition for the uniqueness of a minimal
  tropical basis and a characterization}
\begin{theo}\label{theo:3}
  If a simple matroid $M$ has a minimal tropical basis which contains
  a circuit $C$ such that $\cl(C) \ne C$, then this minimal tropical
  basis is not unique.
\end{theo}

\begin{proof}
  Let $\mathcal{C}'$ be a minimal tropical basis of $M$ and $C \in
  \mathcal{C}'$ such that there exists $f \in \cl(C) \setminus C$. Then
  $C \cup f$ is a double circuit. Let $D_1\cup\ldots \cup D_k \cup \{f\},
  k\ge 2$ be its partition. Define
\[\mathcal{C}'':= (\mathcal{C}' \setminus C) \cup \left\{ (C \cup \{f\}) \setminus D_i \mid 1\le i \le k \right \}.\]
We claim that $\mathcal{C}''$ is a tropical basis and hence it
contains a minimal tropical basis $\mathcal{C}'''$ different from
$\mathcal{C}'$.

To see that, consider a set $X$ for which $C$ is the only element in
$\mathcal{C}'$ such that $C \setminus X = \{d\}$. We may assume that $d \in D_1$. But then
\[\left((C \setminus D_1) \cup \{f\}\right ) \setminus X = \{f\}.\]
Hence $\mathcal{C}''$ is a tropical basis.
\end{proof}

This yields the following characterization of simple matroids having a unique
minimal tropical basis.

\begin{theo}\label{theo:4}
  A simple matroid $M$ has a unique tropical basis if and only if
 \[\bar X \bot \{C \in \mathcal{C} \mid \cl(C)=C\} \Longleftrightarrow \cl(X)=X.\]
\end{theo}
\begin{proof}
  This follows from Proposition~\ref{prop:1} and Theorem~\ref{theo:2}.
\end{proof}

Hence $M$ has a unique minimal tropical basis if the minimal non-empty
sets which are orthogonal to all closed circuits form the family of
cocircuits of $M$.

Note, that in $P_7$ (see ~\cite{oxley} Page 646), which does not
satisfy the assumptions of Theorem~\ref{theo:3}, any set which is not
closed meets one of the three point lines in exactly $2$ points. Hence
its complement is not orthogonal to the set of three point lines. This
shows that $P_7$ has a unique minimal tropical basis by
Theorem~\ref{theo:4}, as was already observed by~\cite{Nak23}.

\end{document}